\documentclass[submission]{dmtcs}
\usepackage{amsmath,url}
\usepackage{amscd}      
\usepackage{longtable}
\usepackage{multicol}
\usepackage{comment}      
 
\def\softO{\tilde{O}}

\def\Q{\mathbb Q} 
\def\S{\mathfrak{S}}

\newtheorem{thm}{Theorem} 
\newtheorem{ex}{Example}

\def\<#1>{\langle#1\rangle}
\let\set\mathbb

\def\threeFtwo#1#2#3#4#5#6{{_3F_2}\biggl(\begin{matrix}
  {#1}\kern.707em {#2}\kern.707em{#3}\\{#4}\kern1em{#5}
\end{matrix}\,\bigg|\,#6\biggr)}
\def\twoFone#1#2#3#4{{_2F_1}\biggl(\begin{matrix}
  {#1}\kern.707em {#2}\\{#3}
\end{matrix}\,\bigg|\,#4\biggr)}
\def\clap#1{\hbox to0pt{\hss#1\hss}}

\def\stepset#1#2#3#4#5#6#7#8{%
  \begin{picture}(20,20)(-10,-10)
    \put(0,0){\ifx1#1\thicklines\let\x\vector\else\thinlines\let\x\line\fi\x(-1,-1){10}}
    \put(0,0){\ifx1#2\thicklines\let\x\vector\else\thinlines\let\x\line\fi\x(-1,0){10}}
    \put(0,0){\ifx1#3\thicklines\let\x\vector\else\thinlines\let\x\line\fi\x(-1,1){10}}
    \put(0,0){\ifx1#4\thicklines\let\x\vector\else\thinlines\let\x\line\fi\x(0,-1){10}}
    \put(0,0){\ifx1#5\thicklines\let\x\vector\else\thinlines\let\x\line\fi\x(0,1){10}}
    \put(0,0){\ifx1#6\thicklines\let\x\vector\else\thinlines\let\x\line\fi\x(1,-1){10}}
    \put(0,0){\ifx1#7\thicklines\let\x\vector\else\thinlines\let\x\line\fi\x(1,0){10}}
    \put(0,0){\ifx1#8\thicklines\let\x\vector\else\thinlines\let\x\line\fi\x(1,1){10}}
  \end{picture}}

\def\stepset#1#2#3#4#5#6#7#8{%
  \raisebox{-7pt}{%
  \setlength{\unitlength}{.9pt}%
  \begin{picture}(20,20)(-10,-10)
    \put(-5,-5){\clap{$\scriptstyle\ifx1#1\bullet\else\cdot\fi$}}
    \put(0,-5){\clap{$\scriptstyle\ifx1#2\bullet\else\cdot\fi$}}
    \put(5,-5){\clap{$\scriptstyle\ifx1#3\bullet\else\cdot\fi$}}
    \put(-5,0){\clap{$\scriptstyle\ifx1#4\bullet\else\cdot\fi$}}
    \put(5,0){\clap{$\scriptstyle\ifx1#5\bullet\else\cdot\fi$}}
    \put(-5,5){\clap{$\scriptstyle\ifx1#6\bullet\else\cdot\fi$}}
    \put(0,5){\clap{$\scriptstyle\ifx1#7\bullet\else\cdot\fi$}}
    \put(5,5){\clap{$\scriptstyle\ifx1#8\bullet\else\cdot\fi$}}
  \end{picture}}}

\def\Stepset#1#2#3#4#5#6#7#8#9{%
  \raisebox{-9pt}{%
  \setlength{\unitlength}{.9pt}%
  \begin{picture}(20,20)(-10,-10)
    \put(-5,-5){\clap{$\scriptstyle\ifx1#1\bullet\else\cdot\fi$}}
    \put(0,-5){\clap{$\scriptstyle\ifx1#2\bullet\else\cdot\fi$}}
    \put(5,-5){\clap{$\scriptstyle\ifx1#3\bullet\else\cdot\fi$}}
    \put(-5,0){\clap{$\scriptstyle\ifx1#4\bullet\else\cdot\fi$}}
    \put(0,0){\clap{$\scriptstyle\ifx1#5\bullet\else\cdot\fi$}}
    \put(5,0){\clap{$\scriptstyle\ifx1#6\bullet\else\cdot\fi$}}
    \put(-5,5){\clap{$\scriptstyle\ifx1#7\bullet\else\cdot\fi$}}
    \put(0,5){\clap{$\scriptstyle\ifx1#8\bullet\else\cdot\fi$}}
    \put(5,5){\clap{$\scriptstyle\ifx1#9\bullet\else\cdot\fi$}}
  \end{picture}}\StepsetB}
\def\StepsetB#1#2#3#4#5#6#7#8{%
  \raisebox{-8pt}{%
  \setlength{\unitlength}{.9pt}%
  \begin{picture}(20,20)(-10,-10)
    \put(-5,-5){\clap{$\scriptstyle\ifx1#1\bullet\else\cdot\fi$}}
    \put(0,-5){\clap{$\scriptstyle\ifx1#2\bullet\else\cdot\fi$}}
    \put(5,-5){\clap{$\scriptstyle\ifx1#3\bullet\else\cdot\fi$}}
    \put(-5,0){\clap{$\scriptstyle\ifx1#4\bullet\else\cdot\fi$}}
    \put(5,0){\clap{$\scriptstyle\ifx1#5\bullet\else\cdot\fi$}}
    \put(-5,5){\clap{$\scriptstyle\ifx1#6\bullet\else\cdot\fi$}}
    \put(0,5){\clap{$\scriptstyle\ifx1#7\bullet\else\cdot\fi$}}
    \put(5,5){\clap{$\scriptstyle\ifx1#8\bullet\else\cdot\fi$}}
  \end{picture}}\StepsetC}
\def\StepsetC#1#2#3#4#5#6#7#8#9{%
  \raisebox{-7pt}{%
  \setlength{\unitlength}{.9pt}%
  \begin{picture}(20,20)(-10,-10)
    \put(-5,-5){\clap{$\scriptstyle\ifx1#1\bullet\else\cdot\fi$}}
    \put(0,-5){\clap{$\scriptstyle\ifx1#2\bullet\else\cdot\fi$}}
    \put(5,-5){\clap{$\scriptstyle\ifx1#3\bullet\else\cdot\fi$}}
    \put(-5,0){\clap{$\scriptstyle\ifx1#4\bullet\else\cdot\fi$}}
    \put(0,0){\clap{$\scriptstyle\ifx1#5\bullet\else\cdot\fi$}}
    \put(5,0){\clap{$\scriptstyle\ifx1#6\bullet\else\cdot\fi$}}
    \put(-5,5){\clap{$\scriptstyle\ifx1#7\bullet\else\cdot\fi$}}
    \put(0,5){\clap{$\scriptstyle\ifx1#8\bullet\else\cdot\fi$}}
    \put(5,5){\clap{$\scriptstyle\ifx1#9\bullet\else\cdot\fi$}}
  \end{picture}}}

\def\sequenceThreeD#1#2#3#4{%
  #2\quad\hbox to.6\hsize{$#3,\dots$\hfill}\quad\textrm{(#4)}
}

\newcounter{mytable}\def\mytable{\noindent\refstepcounter{mytable}\textbf{Table~\arabic{mytable}}}

\begin{document}

 \title{Automatic Classification of Restricted Lattice Walks}

 \author{Alin Bostan\addressmark{1}\thanks{Partially supported by the 
French National Agency for Research 
(ANR Project ``Gecko'') and the Microsoft Research-INRIA  
Joint Centre.} \and Manuel Kauers\addressmark{2}\thanks{Partially supported by the Austrian Science Foundation (FWF) grants P19462-N18 and P20162-N18.}}

\address{\addressmark{1} Algorithms Project, INRIA Paris-Rocquencourt, 78153 Le  
Chesnay, France\\
\addressmark{2} Research Institute for Symbolic Computation, J. Kepler  
University Linz, Austria}

\maketitle

 \begin{abstract}  
We propose an \emph{experimental mathematics approach} 
leading to the computer-driven \emph{discovery} of various conjectures about 
structural properties of generating functions coming from enumeration of restricted lattice walks in 2D and in 3D. 
 \end{abstract}
 \section{Introduction}



 There is a strange phenomenon about the generating functions that count lattice walks restricted to the 
 quarter plane: depending on the choice of the set $\mathfrak{S}\subseteq\{{\swarrow,}\ {\leftarrow,}\ {\nwarrow,}\ 
 {\uparrow,}\ {\nearrow,}\ {\rightarrow,}\ {\searrow,}\ {\downarrow}\}$ of admissible steps, the generating
 function is sometimes rational, sometimes algebraic [but not rational], sometimes D-finite [but not algebraic], and sometimes not even D-finite. 
 This is quite in contrast to the corresponding problem in~1D, where the generating functions invariably
 are algebraic~\cite{BaFl02}.
 Much progress was made recently on understanding why this is so, and only very recently,
 Bousquet-M\'elou and Mishna~\cite{BoMi08} have announced a classification of all the 256 possible
 step sets into algebraic, transcendental D-finite, and non-D-finite cases, together with proofs for the 
 algebraic and D-finite cases and strong evidence supporting the conjectured non-D-finiteness of the others. 

 As usual, a power series $S(t) \in \Q[[t]]$ is called \emph{algebraic} if there exists a bivariate 
 polynomial~$P(T,t)$ in $\Q[T,t]$ such that $P(S(t),t)=0$,
 and transcendental otherwise. Also as usual, a power series $S(t)$ is called D-finite if it satisfies a linear
 differential equation with polynomial coefficients. (Every algebraic power series is D-finite, but not vice versa.)
 At first glance, it might seem easy to prove that a power series is algebraic or D-finite: just come up with an
 appropriate equation, and then verify that the series satisfies this equation. But as far as lattice walks are concerned, 
 most proofs given so far are indirect in that they avoid exhibiting the equation explicitly 
 but merely are satisfied showing its existence.
 This is probably so because the equations appearing in this context are often too big to be dealt with by hand. 

 Nevertheless, it is interesting to know the equations explicitly, because they provide a standard canonical 
 representation for a series, from which lots of further information can be extracted in a straightforward manner. 
 By applying a well-known technique from computer algebra (in modern fashion, cf.~Section~\ref{sec:2}), we have 
 systematically searched for differential equations and algebraic equations that the series counting the 
 walks in the quarter plane satisfy. These are given in Section~\ref{sec:3}. 
 We have also made a first step towards classifying walks in $\set Z^3$ confined to the first octant (cf.~Section~\ref{sec:4}) 
 by considering all step sets $\mathfrak{S}$ with up to five elements, and performed a systematic search for equations of the 
 corresponding series.  More than 2000 hours of computation time have been spent in order to analyze about 3500 different sequences. 

 We do not provide proofs that the equations we found are indeed correct, but the computational evidence in favor of
 our equations is striking. We have no doubt that all the equations we found are correct. 
 In principle, it would be possible to supplement the ``automatically guessed'' equations by computer proofs in 
 a systematic fashion, using techniques that have recently been applied to some special cases~\cite{KaZe08,KaKoZe08,BoKa08}.
 But we found that the computational cost for performing these 
automated proofs
would be 
by far
higher than what was needed for the mere discovery.

\section{Methodology} \label{sec:2}   
To study generating functions for lattice walks, we
follow a classical scheme in experimental mathematics. It is based on the following steps: 
(S1)~computation of high order expansions of generating power series; 
(S2)~guessing differential and/or algebraic equations satisfied by those power series; 
(S3)~empirical certification of the guessed equations (sieving by inspection of their
analytic, algebraic and arithmetic properties); (S4) rigorous proof, based on (exact) polynomial computations.

In what follows, we only explain Steps~(S1), (S2) and (S3). A full description of Step (S4) is given in~\cite{BoKa08}. By way of illustration, we choose an example requiring computations with
human-sized outputs, namely the classical case, initially considered by
Kreweras~\cite{Kreweras65,Bousquet02,Bousquet05}, of walks in the quarter
plane restricted to the step set $\mathfrak{S} =\{{\leftarrow,}\ {\nearrow,}\
{\downarrow}\}$. 

\subsection{Basic Definitions and Facts}
We focus on 2D and 3D lattice walks. The 2D walks that we consider are confined to the quarter plane~$\mathbb{N}^2$, they join the origin of $\mathbb{N}^2$ to an arbitrary point $(i,j) \in \mathbb{N}^2$, and are restricted to 
a fixed subset $\mathfrak{S}$ of the step set $\{{\swarrow,}\ {\leftarrow,}\ {\nwarrow,}\  {\uparrow,}\ {\nearrow,}\ {\rightarrow,}\ {\searrow,}\ {\downarrow}\}$.  
If $f(n;i,j)$ denotes the number of such walks of length~$n$
 (i.e., using $n$ steps chosen from~$\mathfrak{S}$), the sequence $f(n;i,j)$ satisfies the multivariate recurrence with constant coefficients
\begin{equation}    \label{eq:mrec}      
f(n+1;i,j) = \sum_{(h,k) \in \mathfrak{S}} f(n;i-h,j-k)  \quad \textrm{for} \quad n,i,j\geq0.
\end{equation} 
Together with the appropriate boundary conditions
$$f(0;0,0)=1 \quad \textrm{and}\quad f(n;i,j)=0 \quad \text{if $i<0$ or $j<0$ or $n<0$},$$
the recurrence relation~\eqref{eq:mrec} uniquely determines the sequence $f(n;i,j)$. As is customary in combinatorics, we let 
 \[
  F(t;x,y)=\sum_{n \geq 0}\Bigl(\sum_{i,j \geq 0} f(n;i,j)x^i y^j\Bigr)t^n
 \]
be the trivariate generating power series 
of the sequence $f(n;i,j)$. As $f(n;i,j)=0$ as soon as $i>n$ or $j>n$, the inner sum is actually finite, and so we may regard $F(t;x,y)$ as a formal power series in $t$ with polynomial coefficients in $\mathbb{Q}[x,y]$.

Specializing $F(t;x,y)$ to selected values of $x$ and $y$ leads to
various combinatorial interpretations. 
 Setting $x=y=1$ yields the power series $F(t;1,1)$ whose coefficients count the total number of walks with prescribed number of steps 
(and arbitrary endpoint);
 the choice $x=y=0$ gives the series $F(t;0,0)$ whose coefficients count the number of walks returning to the origin;
 setting $x=1$, $y=0$ yields the power series 
whose coefficients count the number of walks ending somewhere on the horizontal axis, etc.

By~\cite[Th.~7]{BoPe00}, multivariate sequences that satisfy recurrences with constant coefficients have moderate growth, and thus their generating series 
are analytic at the origin. The next theorem refines this result 
in our context.    

\begin{thm} \label{theo:Radius}
The following inequality holds
\begin{equation}\label{eq:bound}
f(n;i,j) \leq |\S|^n \quad \text{for all} \quad (i,j,n) \in \set N^3.
\end{equation}
In particular, the power series $F(t;0,0), F(t;1,0), F(t;0,1)$ and $F(t;1,1)$
are convergent in $\mathbb{C}[[t]]$ at $t=0$ and their radius of convergence is at least $1/|\S|$.
\end{thm}

\begin{proof}  
 The total number of unrestricted $n$-step walks starting from the origin is~$|\S|^n$, so
 the number of walks restricted to a certain region is bounded by this quantity. 
 This implies that the coefficient of $t^n$ in $F(t;1,1)$ is at most $|\S|^n$. The bound also applies to the coefficient of $t^n$ in $F(t;\alpha,\beta)$
for $\alpha,\beta\in\{0,1\}$, as these series count walks which are subject to further restrictions.
\end{proof}

\subsubsection{D-finite generating series of walks are $G$-functions}
A power series $S(t)=\sum_{n \geq 0} a_n t^n$ in $\mathbb{Q}[[t]]$ is called a $G$-function\footnote{The usual definition is more general, the coefficients of $S$ can be taken in an arbitrary algebraic number field. For our purposes it is sufficient and convenient to restrict to rational coefficients.} if (a)~it is D-finite; (b)~its radius of convergence in $\mathbb{C}[[t]]$ is positive; 
(c)~there exists a constant $C>0$ such that for all $n \in \mathbb{N}$, the common denominator of $a_0,\ldots,a_n$ is bounded by $C^n$. 

Examples of $G$-functions are the power series expansions at the origin of  $\log(1-t)$ and $(1-t)^\alpha$ for $\alpha \in \mathbb{Q}$. More generally, the Gauss hypergeometric series $_2 F_1(\alpha,\beta,\gamma ;t)$ with rational parameters $\alpha,\beta,\gamma$, is also a $G$-series~\cite{DwGeSu94}. A celebrated theorem of Eisenstein assures that any algebraic power series must be a $G$-function (if $S$ is algebraic, there exists an integer $C \in \mathbb{N}$ such that $a_n C^{n+1}$ is an integer for all $n$.)  The fact that $G$-functions arise frequently in combinatorics was recently pointed out by Garoufalidis~\cite{Garoufalidis08}.   

$G$-functions enjoy many remarkable properties. Chudnovsky~\cite{ChCh85}  
proved that the minimal order differential equation satisfied by a $G$-series must be \emph{globally nilpotent} (see Section~\ref{ssec:pcurv} below for the definition and an algorithmic use of this notion).
By a theorem of Katz and Honda~\cite{Katz70,Honda81}, the global
nilpotence of a differential operator implies that all of its singular points
are \emph{regular singular\/} points with \emph{rational exponents}.
See also~\cite{Andre89,ChambertLoir00,DwGeSu94} for more details on this topic.

\begin{thm}  \label{theo:Gseries}
Let $S(t)$ be one of the power series $F(t;0,0), F(t;1,0), F(t;0,1)$ and $F(t;1,1)$. If $S$ is D-finite, then $S$ is a $G$-series. 
In particular, its minimal order homogeneous linear differential equation is Fuchsian and it has only rational exponents. 
Moreover, the coefficient sequence of $S(t)$ is asymptotically equivalent to a sum of terms of the form $\kappa \rho^n n^\alpha (\log n)^\beta$ for some 
constants $\kappa\in\set R$, $\alpha\in\set Q$, $\rho\in\overline{\set Q}$, and $\beta\in\set N$.
\end{thm}                  

\begin{proof}
     The conditions (a) and (c) in the definition of a $G$-function are clearly satisfied. The only non-trivial point is the fact that the series $S$ has a positive radius of convergence in $\mathbb{C}$. This follows from Theorem~\ref{theo:Radius}.
The Fuchsianity of the minimal equation for $S$, and the rationality of its exponents, follow
by combining the results by Katz, Honda and Chudnovsky cited above.
The claim on the asymptotics of the coefficients of $S(t)$ is a consequence of~\cite[Prop. 2.5]{Garoufalidis08}.
\end{proof}
  
For 3D 
walks, the definitions are analogous. The trivariate power series $F(t;x,y)$ is simply replaced by the generating series $G(t;x,y,z)\in\Q[x,y,z][[t]]$ of the sequence $g(n;i,j,k)$ that counts walks in $\set N^3$ starting at $(0,0,0)$ and ending at $(i,j,k) \in \set N^3$. Note that the appropriate versions of Theorems~\ref{theo:Radius} and \ref{theo:Gseries} hold; in particular, the generating series of octant walks $G(t;1,1,1)$ is 
a $G$-series whenever it is D-finite. 
                               
\subsection{Computing large series expansions}\label{sec:series_exp}   
The recurrence~\eqref{eq:mrec} can be used to determine the value of $f(n;i,j)$ for specific integers $n,i,j \in \set N$. Theorem~\ref{theo:Radius} implies that $f(n;i,j)$ is a non-negative integer whose bit size is at most $O(n)$. If $N \in  \set N$, the values $f(n;i,j)$ for $0\leq n,i,j \leq N$ can thus be computed altogether by a straightforward algorithm that uses $O(N^3)$ arithmetic operations and $\softO(N^4)$ bit operations. (We assume that two integers of bit-size $N$ can be multiplied in $\softO(N)$ bit operations; here, the soft-O notation $\softO(\;)$ hides logarithmic factors.) The memory storage requirement is proportional to~$N^3$.
The same is also true for the truncated power series $F_N = F(t;x,y) \bmod t^N$. 
For our experiments in~2D, we have chosen $N=1000$. With this choice, the computation of the $f(n;i,j)$ is the step which consumes by far the most computation time in our calculations.%
\footnote{We have carried out our computations on various different machines whose main memory ranges from 8\,Gb to 32\,Gb and which are 
 equipped with (multiple) processors all running at about~3GHz.}

\begin{ex}
The Kreweras walks satisfy the recurrence
\begin{equation*}       \label{eq:recKrew} 
 f(n+1,i,j)=f(n,i+1,j)+f(n,i,j+1)+f(n,i-1,j-1)\quad \textrm{for} \quad n,i,j\geq0,
\end{equation*}
which allows the computation of the first terms of the series $F(t;x,y)$
\begin{alignat*}1
  F(t;x,y)&= 1 + x y t + (x^2 y^2+y+x) t^2 + (x^3 y^3+2 x y^2+2 x^2 y+2) t^3 \\
  &\quad{} + (x^4 y^4+3 x^2 y^3+3 x^3 y^2+2 y^2+6 x y+2 x^2) t^4 \\
  &\quad{} + (x^5 y^5+4 x^3 y^4+4 x^4 y^3+5 x y^3+12 x^2 y^2+5 x^3 y+8 y+8 x)
  t^5 + \cdots
 \end{alignat*}
and also the first terms of the generating series $F(t;1,1)$ for the total number of Kreweras walks 
\begin{alignat*}1 
	F(t;1,1)={}& 
1+t+3t^2+7t^3+17t^4+47t^5+125t^6+333t^7+939t^8+2597t^9+\\
& 7183t^{10}+20505t^{11}+57859t^{12}+163201t^{13}+469795t^{14}+\cdots
 \end{alignat*} 
\end{ex}

In the 3D case,  the values $g(n;i,j,k)$ for $0\leq n,i,j,k \leq N$ can be computed in $O(N^4)$ arithmetic operations, $\softO(N^5)$ bit operations and 
$O(N^4)$ memory space. In practice, we found that computing $G \bmod t^N$ with $N=400$ is feasible.

 \subsection{Guessing}\label{sec:2.2}

Once the first terms of a power series are determined, our approach is to search systematically for candidates of linear differential equations or of algebraic equations which the series may possibly satisfy. This technique is classical in computer algebra and mathematical physics, see for example~\cite{BrGu90,SaZi94,mallinger96}. Differential and algebraic guessing procedures are 
available in some computer algebra systems like Maple and Mathematica.
  
\subsubsection{Differential guessing}\label{sec:deguess}
If the first $N$ terms of a power series $S \in \Q[[t]]$ are available, one can search for a differential equation satisfied by $S$ at precision $N$, that is, for an element $\mathcal{L}$ in the Weyl algebra $\Q[t]\<D_t>$ of  differential operators in the derivation $D_t = \frac{d}{dt}$ with polynomial coefficients in~$t$, such that          
\begin{equation} \label{eq:diffeq}
	\mathcal{L}(S) = c_r(t) S^{(r)}(t) + \cdots + c_1(t)S'(t) + c_0(t) S(t) = 0 \bmod t^N.
\end{equation}                                                       
Here, the coefficients $c_0(t),\ldots, c_r(t)\in \Q[t]$ are not simultaneously zero, and their degrees are bounded by a prescribed integer~$d\geq 0$.
By a simple linear algebra argument, if $d$ and $r$
are chosen such that $(d+1)(r+1) > N$, then such a differential equation always
exists. 
On the other side, if $d,r$ and $N$ are such that $(d+1)(r+1) \ll N$, the equation~\eqref{eq:diffeq} translates into a highly over-determined linear system, so it has no reason to possess a non-trivial solution.

The idea is that if the given power series $S(t)$ happens to be D-finite, then for a sufficiently large~$N$, a differential equation of type~\eqref{eq:diffeq} (thus satisfied a priori only at precision~$N$) will provide a differential equation which is really satisfied by $S(t)$ in $\Q[[t]]$ (i.e., at precision infinity).
In other words, the D-finiteness of a power series can be (conjecturally) recognized using a finite amount of information.

Given the values $d,r,N$, and the first $N$ terms of the series~$S$, a
candidate differential equation of type~\eqref{eq:diffeq} for $S$ can be computed by Gaussian elimination in $O(N^3)$ arithmetic operations and $\softO(N^4)$ bit operations. Actually, a modular approach is preferred to a direct Gaussian elimination over~$\Q$. Precisely, the linear algebra step is performed modulo several primes~$p$, and the results (differential operators modulo~$p$) are recombined over $\Q$ via rational reconstruction based on an effective version of the Chinese remainder theorem. (See~\cite{kauers09a} for an implementation of this technique in Mathematica.)

If no differential equation is found, this definitely rules out the possibility that a differential equation of order~$r$ and degree~$d$ exists.
This does not, however, imply that the series at hand is not D-finite. It may still be that the series satisfies a differential equation of order higher
than~$r$ or an equation with polynomial coefficients of degree exceeding~$d$.

Asymptotically more efficient guessing algorithms exist, based on fast Hermite-Pad\'e approximation~\cite{BeLa94} of the vector of (truncated) power series $[S,S', \ldots, S^{(r)}]$; they have arithmetic complexity quadratic or even softly-linear in $N$. Such sophisticated algorithms were not needed to obtain the results of this paper, but they have provided crucial help in the treatment of examples of critical sizes (e.g. guessing with higher values of $d,r,N$ and/or over a parametric base field like $\Q(x)$ instead of~$\Q$) needed for the proof 
in~\cite{BoKa08}.

\begin{ex} [continued]
 $N=100$ terms of the generating series  $F(t;1,1)$ of the total number of Kreweras walks  are sufficient to conjecture that $F(t;1,1)$ is D-finite,
 since it verifies the differential equation $\mathcal{L}_{1,1}(F(t;1,1)) = 0 \bmod t^N$, where
\begin{alignat}1    
\mathcal{L}_{1,1}&=  4 t^2 (t+1) (3 t-4) (3 t-1)^3 (9 t^2+3 t+1) D_t^4\notag\\
&\quad{} + 2 t (3 t-1)^2 (2916 t^5-1296 t^4-3564 t^3-477 t^2-93 t+52)D_t^3\notag\\
&\quad{} + 3 (3 t-1) (29808 t^6-26244 t^5-28440 t^4+2754 t^3+431 t^2+448 t-40) D_t^2\label{eq:diffeqKrew}\\
&\quad{} +6 (68040 t^6-88452 t^5-37206 t^4+16758 t^3+954 t^2+253 t-126) D_t\notag\\
&\quad{} +18 (6480 t^5-8856 t^4-3078 t^3+714 t^2+211 t+2).\notag
\end{alignat}  
Thus, with high probability, $F(t;1,1)$ verifies the differential equation $\mathcal{L}_{1,1}(F(t;1,1)) = 0$. 
\end{ex} 

Sometimes (see Section~\ref{ssec:pcurv}) one needs to guess the minimal-order differential equation $\mathcal{L}_{\min}(S) = 0$ satisfied by the given generating power series. Most of the time, the choice $(d,r)$ of the target degree and order does not lead to this minimal operator. 
Worse, it may even happen that the number of initial terms $N$ is not large enough to 
allow the recovery of~$\mathcal{L}_{\min}$, while these $N$ terms suffice to guess non-minimal order operators. (The explanation of why such a situation occurs systematically was given 
in~\cite{BoChLeSaSc07}, for the case of differential equations satisfied by algebraic functions.)  A good heuristic is to compute several non-minimal operators and to take their greatest common right divisor; generically, the result is exactly $\mathcal{L}_{\min}$.

As a final general remark, let us point out that a power series satisfies a linear differential equation if and only if its coefficients satisfy a linear
recurrence equation with polynomial coefficients. A recurrence equation can be computed either from a differential equation, or it can be guessed from
scratch by proceeding analogously as described above for differential equations.

\begin{ex}[continued] 
$N=100$ terms of the series $S(t) = F(t;1,1)$ suffice to guess that its coefficients satisfy the 
order-6 recurrence 
\begin{alignat*}1
& 2 (n+6) (n+7) (2 n+13) (7 n+34) u_{n+6}  
- (n+6) (140 n^3+2402 n^2+13687 n+25843) u_{n+5} \\
& + 3 (28 n^4+626 n^3+5123 n^2+18281 n+24070) u_{n+4}\\
& - 18 (n+4) (28 n^3+311 n^2+897 n+304) u_{n+3}  
+ 108 (n+3)(35 n^3+443 n^2+1787 n+2309) u_{n+2} \\
& - 324 (n+2)(7 n^3+90 n^2+382 n+545) u_{n+1} 
- 972 (n+1) (n+2) (n+4) (7 n+41) u_n = 0.
\end{alignat*}
\end{ex}

                                     
\subsubsection{Algebraic guessing} \label{sec:aeguess}
If the first $N$ terms of a power series $S \in \Q[[t]]$ are available, one can also search for an algebraic equation satisfied by $S$ at precision~$N$, that is, for a bivariate polynomial $P(T,t)$ in $\Q[T,t]$ such that         
\begin{equation} \label{eq:algeq}
	P(S(t),t) = c_r(t) S(t)^{r} +\cdots + c_1(t) S(t) + c_0(t) = 0 \bmod t^N. 
\end{equation} 
A similar discussion shows that candidate algebraic equations of type~\eqref{eq:algeq} for $S$ can be ``guessed'' by performing either Gaussian elimination or Hermite-Pad\'e approximation on the vector $[1,S, \ldots, S^{r}]$, followed by a gcd computation in $\Q[T,t]$ applied to two (or more) different guesses.

\begin{ex}[continued] 
$N=100$ terms of the series $S(t) = F(t;1,1)$ counting the total number of Kreweras walks suffice to guess that $F(t;1,1)$ is very probably algebraic, namely solution of the bivariate polynomial   
\begin{alignat}1 
P_{1,1}(T,t)={}&t^5 (3 t-1)^3 T^6+6 t^4 (3 t-1)^3 T^5+t^3 (3 t-1) (135 t^2-78 t+14) T^4\notag\\
   &\qquad{}+4 t^2 (3 t-1) (45 t^2-18 t+4) T^3+t (3 t-1) (135 t^2-26 t+9) T^2\label{eq:algeqKrew}\\
   &\qquad{}+2 (3 t-1) (27 t^2-2 t+1) T+43 t^2+t+2.\notag
\end{alignat}
\end{ex}

\subsection{Empirical certification of guesses}\label{sec:1bis} 
Once discovered a differential equation~\eqref{eq:diffeq} or an algebraic equation~\eqref{eq:algeq} that the power series $S(t)$ seems to satisfy, we inspect several properties of these equations, in order to provide more convincing evidence that they are correct.   These properties have various natures: some are computational features (moderate bit sizes), others are algebraic, analytic and even arithmetic properties. We check them systematically on all the candidates; if they are verified, as in the Kreweras example, this offers striking evidence that the guessed equations are not artefacts.

\subsubsection{Size sieve: Reasonable bit size} The differential equation~\eqref{eq:diffeq} 
has typically much lower bit size than a differential equation produced by the same guessing procedure applied to the same order, degree and precision, but to an arbitrary series having coefficients of bit-size comparable to that of $S(t)$. A similar observation holds for the algebraic equation~\eqref{eq:algeq}.

\begin{ex}[continued]
If we perturb the coefficients of $S(t)=F(t;1,1)$ by just adding 
a random integer between $-100$ and $100$
to each of its coefficients, then the differential guessing procedures at order $r=4$, degree $d=9$ and precision $N=100$ will either give no result (the over-determined system approach) or produce fake candidates (the Hermite-Pad\'e approach) with polynomial coefficients in $t$, whose coefficients in~$\mathbb{Q}$ have numerators and denominators of about 500 decimal digits each, instead of 4 digits for $\mathcal{L}_{1,1}$.
\end{ex}


\subsubsection{Algebraic sieve: High order series matching}
The equations~\eqref{eq:diffeq} and~\eqref{eq:algeq} were obtained starting from $N$ coefficients of the power series $S(t)$. They are therefore satisfied a priori only modulo $t^N$. We compute more terms of $S(t)$, say $2N$, and check whether the same equations still hold modulo $t^{2N}$. If this is the case, chances increase that the guessed equations also hold at infinite precision.

\subsubsection{Analytic sieve: Singularity analysis}
By Theorem~\ref{theo:Gseries}, the minimal order operators for power series like $S(t)=F(t;0,0)$ and $S(t)=F(t;1,1)$ must have only regular singularities (including the point at infinity) and their exponents must be rational numbers.

\begin{ex}[continued]
The differential operator $\mathcal{L}_{1,1}$ is Fuchsian. 
Indeed, a (fully automated) local singularity analysis shows that the set 
of its singular points 
$\left \{-1, 0, \infty, \tfrac13, \tfrac43, -\tfrac16(1 \pm i\sqrt3) \right\}$
is formed solely of regular singularities. Moreover, the indicial polynomials  
of $\mathcal{L}_{1,1}$ 
are, respectively:
$t(t-1)(t-2)(2t-1), 
t(t-1)(2t+1)(t+1), 
(t-5)(t-1)(t-2)(t-4), 
(t+1)t(4t-1)(4t+1), 
t(t-1)(t-2)(t-4)$,       
and $t(t-2)(2t-3)(t-1)$. Their roots are the rational exponents of the singularities.

\end{ex}
\subsubsection{Arithmetic sieve: $G$-series and global nilpotence}   \label{ssec:pcurv}
Last, but not least, we check an arithmetic property of the guessed differential equations by exploiting the fact that those expected to arise in our combinatorial context are very special.          
                      
Indeed, by a theorem due to the Chudnovsky brothers~\cite{ChCh85}, the minimal order differential operator $\mathcal{L} \in \Q[t]\langle D_t \rangle$ killing a $G$-series enjoys a remarkable arithmetic property: $\mathcal{L}$ is \emph{globally nilpotent}. By definition, this means that for almost every prime number~$p$ (i.e., for all with finitely many exceptions), there exists an integer $\mu \geq 1$ such that the remainder of the Euclidean (right) division of $D_t^{p \mu}$ by $\mathcal{L}$ is congruent to zero modulo~$p$~\cite{Honda81,Dwork90}.
              
 {}From a computational view-point, a fine feature is that the  
nilpotence modulo~$p$ is checkable.
If $r$ denotes the order of $\mathcal{L}$, let $M_p$ be the \emph{$p$-curvature matrix of $\mathcal{L}$}, defined as the $r \times r$  
matrix with entries in $\Q(t)$ whose $(i,j)$ entry is the  
coefficient of $D_t^{j-1}$ in the remainder of the Euclidean (right)  
division of $D_t^{p+i-1}$ by $\mathcal{L}$. Then, $\mathcal{L}$ is  
nilpotent modulo~$p$ if and only if
the matrix~$M_p$ is nilpotent modulo~$p$~\cite{Dwork90,Schmitt93}.

In combination with Theorem~\ref{theo:Gseries}, this yields a fast
algorithmic filter: as soon as we guess a candidate differential
equation satisfied by a generating series which is suspected to be a $G$-series (e.g. by $F(t;1,1)$), we check
whether its $p$-curvature is nilpotent, say modulo the first 50 primes
for which the reduced operator $\mathcal{L} \bmod p$ is
well-defined. If the $p$-curvature matrix of $\mathcal{L}$ is
nilpotent modulo $p$ for all those primes $p$, then the guessed
equation is, with very high probability, the correct one.
                               
We push even further this arithmetic sieving. A famous conjecture,
attributed to Grothendieck, asserts that the differential equation
$\mathcal{L}(S)=0$ possesses a basis of \emph{algebraic solutions\/}
(over $\mathbb{Q}(x)$) if and only if its $p$-curvature matrix $M_p$ \emph{is zero modulo~$p$\/} for almost all primes $p$. Even if the conjecture is, for the
moment, fully proved only for order one operators and partially in the
other cases~\cite{ChambertLoir00}, we freely use it as an oracle to
detect whether a guessed differential equation has a basis of
algebraic solutions.  For instance, the computation of the
$p$-curvature of an order 11 differential operator with polynomial
coefficients of degree 96 in $t$, was one of the key points in our
discovery~\cite{BoKa08} that the trivariate generating
function for Gessel walks is algebraic.

\begin{ex}[continued]
The $5$-curvature matrix $M_5(t)$ of the  differential operator~$\mathcal{L}_{1,1}$ in~\eqref{eq:diffeqKrew} has the form $\frac{1}{d(t)}{\tilde{M_5}(t)}$, where $d(t)= 
(3t - 1)^7 t^6 (t+1)^5 (9t^2 + 3 t + 1)^5 (3t - 4)$ and ${\tilde{M_5}(t)}$ is a $4\times 4$ 
matrix with polynomial entries in $\mathbb{Q}[t]$ of degree at most~$27$.
The characteristic polynomial $\chi_{M_5}$ of $M_5$ reads           
\begin{alignat*}1    
T^4 + 
	\frac{3 \cdot 5}{2^5} \, N_{3}(t) \, t^5 \,(3t-1)^{10} \,T^3 
+\frac{3^3 \cdot 5}{2^{10}}  \, N_{2}(t) \, (3t-1)^5 \,T^2 + 
	\frac{3^5 \cdot 5^2 \cdot 7}{2^7} \, N_{1}(t) \, T + 
	\frac{3^9 \cdot 5^3 \cdot 7^2}{2^3} \, N_{0}(t),
\end{alignat*}  
where $N_{0},N_{1},N_{2},N_{3}$ are irreducible polynomials in $\mathbb{Z}[t]$, of degree, respectively, $21, 26, 26, 21$ and with coefficients having at most 20 decimal digits.                               

\smallskip The polynomial $\chi_{M_5}$ obviously equals $T^4$ modulo $p=5$, so the $5$-curvature of $\mathcal{L}_{1,1}$ is nilpotent (but not zero\footnote{Modulo $5$, the curvature matrix $M_5(t)$ has $T^2$ as minimal polynomial.}) modulo $5$. In fact, for all the primes $7\leq p < 100$, the $p$-curvature matrix of $\mathcal{L}_{1,1}$ is also nilpotent modulo $p$; it is even zero modulo~$p$. Under the assumption that Grothendieck's conjecture is true, this indicates that~$\mathcal{L}_{1,1}$ admits a basis of algebraic solutions, and so provides independent evidence that also $S(t) = F(t;1,1)$ is algebraic. 
\end{ex}

 \section{Empirical Results in 2D}\label{sec:3}


In this section, we consider the total number of walks only, i.e., the generating function~$F(t;1,1)$. 
Because of symmetries, the 256 possible step sets give rise to 92 different
sequences only. By inspection of the first $N=1000$ terms, we found that 
36 of them appear to be D-finite: 19 are algebraic and 17 are transcendental. 
The D-finite step sets, together with the sizes of the equations we discovered,
are listed in Table~\ref{tab:2d} in the appendix.
(There, and below, step sets are represented by compact pictograms, e.g.  ${\stepset01010001}$ for $\mathfrak{S} =\{{\leftarrow,}\ {\nearrow,}\
{\downarrow}\}$.)




 \subsection{Combinatorial Observations}

 Our classification matches the results of Bousquet-M\'elou and 
 Mishna~\cite{BoMi08}: for every sequence they prove D-finite
 our software found a recurrence and a differential equation, and whenever
 a series is algebraic indeed, our programs recognized it. 
 Moreover, we found no recurrence or differential equation for any step
 set conjectured non-D-finite by Bousquet-M\'elou and Mishna. 
 This strengthens the evidence in favor of the conjectured non-D-finiteness
 of these cases. 

 \subsection{Algebraic Observations}

 All but two of the minimal polynomials of the algebraic series share the property that they define a curve of genus~0.
 As a consequence, there exists a rational parametrization in all these cases. 
 For example, for the Kreweras step set \smash{\stepset01010001}, the minimal polynomial $P_{1,1}$ given in~\eqref{eq:algeqKrew} 
 defines a curve parameterized by
 \[
  T(u) = \frac{(u^2+24 u+151) a(u)}
            {(u+9)(u^2+24u+147)}
    \quad\text{and}\quad
  t(u)  = \frac{2}{a(u)},
 \]
 where $a(u) = \bigl(u^6+66 u^5+1827 u^4+27180 u^3+229431 u^2+1042866 u+1995717\bigr)\big/(u+11)(u^2+22 u+125)^2$, i.e., for these rational functions we have
 \[
   P_{1,1}(T(u),t(u))=0.
 \]
 The two algebraic series that do not admit a rational parametrization belong to the step sets
 $\smash{\stepset10001010}$ (reverse Kreweras) and $\smash{\stepset11000011}$ (Gessel's). Their genus is~1.

 Another feature of the series which we found to be algebraic is that they all admit closed forms in terms of (nested) radical expressions. 
 For example, for the Kreweras step set, we find that $F(t;1,1)$ is equal to 
 \[
    -\frac{1}{t} + 
    \sqrt{\frac{\left(i-\sqrt{3}\right) \left(216 t^3+1\right)
    \left(t-3 t^2\right)^2-2 i t (36 t^2 -15 t+1) a(t)+\left(i+\sqrt{3}\right) a(t)^2}{6it^3 (3 t-1)^3 a(t)}}
 \]
 where $i=\sqrt{-1}$ and $\displaystyle
  a(t)=\sqrt[3]{24 \sqrt{3t^9 (3 t-1)^9 \left(9 t^2+3
    t+1\right)^3}-t^3 (3 t-1)^3 \left(5832 t^6+540 t^3-1\right)}$.
 Such representations can be found by appealing to the built-in equation solvers of Maple and Mathematica
 applied to the equation~$P_{1,1}=0$.
 Both features are remarkable because, among all algebraic power series, 
 only a few are rationally parameterizable or expressible
 in terms of radicals. 


 Also the transcendental D-finite series appear to have some 
special properties.
 Being D-finite, these series are annihilated by some linear differential operator 
 \[
  \mathcal{L} = c_0(t) + c_1(t)D_t + \cdots + c_r(t)D_t^r\in \set Q[t]\<D_t>.
 \]
According to the DFactor command from Maple's DEtools package, all the operators can be factorized into a  
product of one irreducible operator of order~2 and several operators of order~1. As all the operators are globally nilpotent,  
so are all their factors~\cite{Dwork90,DwGeSu94}.
                                                        
We can therefore expect that every solution of these factors can be  
written as a sum of terms of the form
 \begin{alignat}1 \label{eq:weakpullback}
	R(t)^{\delta} \cdot \twoFone{\alpha} {\beta}{\gamma}{Z(t)},
 \end{alignat} 
where $R$ and $Z$ are rational functions in $\set Q(t)$ and $\alpha,\beta, \gamma, \delta$ are rational numbers.
Indeed, Dwork~\cite[Item 7.4]{Dwork90} has conjectured that any  
globally nilpotent second order differential equation has either  
algebraic solutions or 
is gauge equivalent to a weak pullback of a Gauss hypergeometric differential equation with rational parameters.       
This conjecture was disproved by Krammer~\cite{Krammer96} and recently by Dettweiler and Reiter~\cite{DeRe08};
the counter-examples given in these papers require involved tools in  
algebraic geometry (arithmetic triangle groups, systems associated to periods of Shimura curves,  \ldots)

We are therefore in a win-win situation: either the second order operators  
appearing as factors of our operators 
admit only solutions 
which are indeed sums of 
terms of the form~\eqref{eq:weakpullback}, or there is a simple combinatorial counter-example  
to Dwork's conjecture.  Let us illustrate this on one of the most simple examples, the step set \smash{\stepset10100101}. 
 We find here the differential operator 
 \begin{alignat*}1
   & 4 (32 t^2-12 t-1)
   +4 (8 t-1) (20 t^2-3 t-1) D_t+t (4 t-1) (112 t^2-5) D_t^2
   +t^2 (4 t-1)^2 (4 t+1) D_t^3
 \end{alignat*}
 which Maple factors into 
 \[
  \bigl(2(192t^3-56t^2-6t+1)+4(24t^2-1)(4t-1)tD_t + (4t-1)^2(4t+1)t^2D_t^2\bigr)\bigl(1/t+D_t\bigr).
 \]
 With the help of Maple's built-in differential equation solver (the dsolve command), 
 it can be found that the differential operator gives rise to the representation 
 \[
   F(t;1,1)=-\frac1{4t}+\Bigl(1+\frac1{4t}\Bigr)\twoFone{1/2}{1/2}1{16t^2}.
 \]
 (Incidentally, this solution can also be expressed in terms of elliptic functions.)
 We believe that all the transcendental D-finite generating functions for any step set admit a representation as 
 (a nested integral of) such an expression. 
 The solvers of Maple and Mathematica, however, are able to discover such a representation only in the simplest cases. 
 (Note that at present, no complete algorithm is known that is capable of finding general pullback representations.)

 \subsection{Analytic Observations}

By Theorem~\ref{theo:Gseries}, all the coefficient sequences grow like~$\kappa n^\alpha\rho^n\log(n)^\beta$ for some constants $\kappa,\rho,\alpha,\beta$
(we only care about the dominant part of their asymptotic expansions).
 From the differential equation or the recurrence equation, we can determine $\rho$, $\alpha$, and $\beta$ exactly as roots of characteristic
 polynomials and indicial equations, respectively. (See~\cite{wimp85,flajolet08} on how this is done.)
 We find that $\beta=0$ in all cases. 
 Knowing the recurrence, we can also compute easily tens of thousands of sequence 
 terms. With the help of convergence acceleration techniques~\cite{brezinski78} applied to so many terms, it is possible to determine 
 the remaining constant~$\kappa$ to an accuracy of thirty digits or more.
 With that many digits, it makes sense to search systematically for potential exact 
 expressions of these constants using Plouffe's inverter~\cite{plouffe08} and/or 
 algorithms like LLL and PSLQ~\cite{bailey07}. We 
actually found ``closed form'' expressions 
 for all these constants. They are included in Table~\ref{tab:2d} in the appendix.

  By Theorem~\ref{theo:Radius}, the numbers $\rho$ are bounded by the cardinality of the step set~$\mathfrak{S}$.
  It turns out that $\rho=|\mathfrak{S}|$ unless the vector sum of the elements of the step set points 
  outside the first quadrant. In these cases, $\rho$~is an algebraic number of degree~2 (e.g., $\rho=1+2\sqrt2$ for the step  set \smash{\stepset00010111}). 
For~$\alpha$, we found only non-positive numbers. 
  Note that $\alpha$ being a negative integer implies that the corresponding 
  series is transcendental~\cite{flajolet87}.

  All the constants $\kappa$ have the form 
  $
    u \rho^{e_0}\phi_1^{e_1}\phi_2^{e_2}\cdots\phi_r^{e_r},
  $
  where the $\phi_i$ are usually small integers, the
  $e_i$ are rational numbers, and $u$ is $1/\pi$ if $F(t;1,1)$ is transcendental, and $1/\Gamma(\alpha+1)$ if $F(t;1,1)$ is algebraic.

  There are some cases where the $\phi_i$ are not integers. 
  Among them, very strange is only the case of the step set~
\smash{\stepset11110111},
for which we found $r=1$, $e_0=7/2, e_1=1/2$, $\rho=2+2\sqrt{6}$ and  
$\phi_1=(1137+468 \sqrt6)/152000$.
  This last number may look like a guessing artefact at first glance, but we trust in its correctness, because the number of correct digits exceeds by far the 
  number of correct digits to be expected from an artefact. 

 \section{Empirical Results in 3D}\label{sec:4}


We have investigated walks in three dimensions confined to the first octant with 
step sets of up to five elements. A priori, there are 83682 such step sets, and 
they give rise to 3334 different sequences. Of those, we have computed the first
$N=400$ terms 
of the generating function $G(t;1,1,1)$ of general walks,
and searched for potential differential equations, algebraic 
equations, and recurrence equations. We found that 134 sequences appear to be 
D-finite, and among those, 50~appear to be algebraic. 





 \subsection{Combinatorial Observations}

For some of the sequences, it can be realized that their D-finiteness or 
algebraicity is a consequence of the D-finiteness or algebraicity of a certain
2D walk. For example, the sequence corresponding to the step set
\begin{alignat*}1
 \sequenceThreeD{1485}{\Stepset00000000000011010000010010}{1, 4, 17, 75, 339, 1558, 7247, 34016, 160795, 764388}{A026378}
\end{alignat*}
is readily seen to be D-finite, since it may be regarded as a variation of the 
2D step set $\smash{\stepset00011010}$ in which the step $\uparrow$ appears in two copies
and empty steps are allowed. 
(Here and below, a three dimensional step set is depicted in three separate 
slices: first the arrows tops of the forms $(x,y,-1)$, then $(x,y,0)$, 
then $(x,y,1)$. For example, the step set above is
$\{(-1,0,0),\ (0,1,0),\ (1,0,0),\ (0,0,1),\ (0,1,1)\}$. The given numbers are the first coefficients
in the expansion of~$G(t;1,1,1)$.)

Discarding these cases from consideration, we are left with 35 different 
sequences whose generating series appear to be D-finite; among those, three appear to be
algebraic. Their step sets are given in the appendix. 

 We were not able to find an equation for the step set
 \[
  \sequenceThreeD{1308}{\Stepset00100010000000000100000001}{1, 1, 4, 7, 28, 70, 280, 787, 3148, 9526, 38104}{A149080}
 \]
 which is symmetric about all three axes, not even with 800 terms instead of~400. 
 Also the step set
 \[
  \sequenceThreeD{1357}{\Stepset00001000001010000000000001}{1, 1, 4, 13, 40, 136, 496, 1753, 6256, 22912, 85216}{A149424}
 \]
 which enjoys a rotational symmetry about the middle line of the first octant,
 and which may be viewed as a three dimensional analogue of Kreweras's step
 set, appears to be non-D-finite, even when 800 terms are taken into account.

 For walks in the quarter plane, 
it is conjectured in~\cite[Section~3]{Mishna07} that 
D-finiteness is preserved under reversing arrows, i.e., the generating function for a step set $\mathfrak{S}$ 
 is D-finite if and only if the generating function for the step set $\mathfrak{S}'$ is, when $\mathfrak{S}'$ is obtained from $\mathfrak{S}$
 by reversing all arrows. 
Our computations do not suggest that this criterion also applies in~3D. Among the 134 sequences we found D-finite, there are 42 which correspond to step 
 sets in $\mathfrak{S}$ for whose counterpart in $\mathfrak{S}'$ we were not able to find an equation. 
 Among those, there are some which satisfy only very large equations, so that chances are that they
 remain D-finite upon reversing arrows, but with equations which are too large for us to find. 
 Others satisfy quite small equations, for example the sequence A026378 whose step set is given
 above. 


 \subsection{Algebraic Observations}\label{sec:4.2} 

 As in the 2D case, it turns out that most of the minimal polynomials of the
 algebraic series define curves of genus~0, which therefore can be rationally
 parameterized.  There are twelve cases of genus~1, these are elliptic curves.
 Some of them turn out to be isomorphic (over $\overline{\mathbb{Q}}$). 
For example, those corresponding to the cases
 \begin{alignat*}1
  \sequenceThreeD{1257}{\Stepset00000000001010001000100000}{1, 1, 4, 11, 32, 110, 360, 1163,4112,14066,47848}{A149232}\\
  \sequenceThreeD{1602}{\Stepset00000000001010001000110000}{1, 2, 7, 27, 105, 426, 1787, 7590,32633,142152,624659}{A150591}\\
  \sequenceThreeD{2843}{\Stepset00000000001010001000100001}{1, 2, 10, 40, 176, 808, 3720, 17152,81440,384448}{A151023}
 \end{alignat*}
 all have 1728 as $j$-invariant. They originate from the 2D Kreweras walks. 
 Most interestingly, there are also three step sets originating from the 2D reverse
 Kreweras walks (\smash{\stepset10001010}) for which the genus is 5~(!). 


 For the transcendental series, we could observe the same phenomenon as in~2D:
 all the operators factor as a product of a single irreducible operator of order
 two and several operators of order one. We therefore expect again that all
 these series admit a representation as a hypergeometric pullback. As an example,
 the generating function $G(t;1,1,1)$ of the sequence 
 \[
   \sequenceThreeD{108}{\Stepset00000100000010010010000000}{
     1, 1, 2, 4, 10, 25, 70, 196, 588, 1764
    }{A005817}
 \]
 can be written in the form 
 \begin{alignat*}1
 \frac{4t^2-2t+1}{48t^3}\int_0^t &x^{-2}(4x^2-2x+1)^{-2}
   \Bigl((1-16x^2)(48x^3-12x^2-8x-1)\twoFone{1/2}{1/2}{1}{16x^2}\\
   &\qquad\qquad-(112x^3-4x^2-8x-1)\twoFone{-1/2}{1/2}{1}{16x^2}\Bigr)dx.
 \end{alignat*}
 This representation was found by Mark van Hoeij. It is beyond the scope of the
 standard tools of Maple or Mathematica.

 \subsection{Analytic Observations} 

 Also concerning asymptotics, similar remarks apply as in~2D. All coefficient sequences grow
 like $\kappa n^\alpha\rho^n$ for some constants $\kappa,\alpha,\rho$, where
 $\rho$ is an integer or an algebraic number of degree~2 and $\alpha$ is a
 non-positive number.
 We have not gone through the laborious task of determining
 the constants~$\kappa$.

\bigskip\noindent\textsf{\itshape Acknowledgments.} 
We are grateful to the anonymous
referees for their pertinent suggestions which helped to improve the presentation of this paper.     
We thank Pierre Nicod\`eme and Bruno Salvy, who carefully read a preliminary version and made several useful comments. We also thank Mark van Hoeij for computing the example pullback in 3D given in Section~\ref{sec:4.2}.

 \bibliographystyle{plain}
 \makeatletter
 \def\@openbib@code{\itemsep=-3pt}
 \makeatother
 \bibliography{main}

 \clearpage
 \section*{Appendix}

\normalsize
\mytable\label{tab:2d} D-finite series and their step sets in~2D.
The equation sizes columns refer to (minimal) recurrence equation, differential equation, and algebraic
equation, respectively. Example: The series $F(t;1,1)$ for Kreweras walks (A151265) satisfies
a differential equation of order~4 with polynomial coefficients of degree~9 and an algebraic
equation $P(F(t;1,1),t)=0$ for a polynomial $P(T,t)$ of degree~6     
in~$T$
and~8 in~$t$.
The coefficient sequence of $F(t;1,1)$ satisfies a recurrence
equation of order~6 with polynomial coefficients of degree~4.  
The labels used in the columns ``OEIS Tag'' are taken from Sloane's On-Line Encyclopedia of 
Integer Sequences \url{http://www.research.att.com/~njas/sequences/}. 
Constants in the asymptotics columns are abbreviated 
$A=1+\sqrt{2}, B=1+2 \sqrt{2}, C=1+\sqrt{3}, D=1+2 \sqrt{3}, E=\sqrt{6(379+156\sqrt6)}\ (!), F=1+\sqrt{6}$.

\medskip

\def\itemA#1#2#3#4#5#6#7#8{#2 & #3 & #5 & #6 & #7 &\rule[-.5em]{0pt}{2.35em}\kern-1.2pt$#8$\kern-1.2pt&\global\let\item\itemB}
\def\itemB#1#2#3#4#5#6#7#8{#2 & #3 & #5 & #6 & #7 &\kern-1.2pt$#8$\kern-1.2pt\\\global\let\item\itemA}
\let\item\itemA
\footnotesize
\begin{longtable}{@{}c|c|c|c|c|c||c|c|c|c|c|c@{}}
 OEIS Tag & Steps & \multicolumn{3}{|c|}{Equation sizes} & Asymptotics & OEIS Tag & Steps & \multicolumn{3}{|c|}{Equation sizes} & Asymptotics \\\hline \endhead
 \item{1}{A000012}{\stepset00000001}{1, 1, 1, 1, 1, 1, 1, 1}{1, 0}{1, 1}{1, 1}{1}%
 \item{2}{A000079}{\stepset00000011}{1, 2, 4, 8, 16, 32, 64, 128}{1, 0}{1, 1}{1, 1}{2^n}%
 \item{2}{A001405}{\stepset00000101}{1, 1, 2, 3, 6, 10, 20, 35}{2, 1}{2, 3}{2, 2}{\dfrac{\sqrt{2}}{\Gamma (\frac{1}{2})}\dfrac{2^n}{\sqrt{n}}}%
 \item{3}{A000244}{\stepset00001011}{1, 3, 9, 27, 81, 243, 729, 2187}{1, 0}{1, 1}{1, 1}{3^n}%
 \item{3}{A001006}{\stepset00110010}{1, 1, 2, 4, 9, 21, 51, 127}{2, 1}{2, 3}{2, 2}{\dfrac{3 \sqrt{3}}{2\Gamma (\frac{1}{2})}\dfrac{3^n}{n^{3/2}}}%
 \item{3}{A005773}{\stepset00000111}{1, 2, 5, 13, 35, 96, 267, 750}{2, 1}{2, 3}{2, 2}{\dfrac{\sqrt{3}}{\Gamma (\frac{1}{2})}\dfrac{3^n}{\sqrt{n}}}%
 \item{3}{A126087}{\stepset00010101}{1, 1, 3, 5, 15, 29, 87, 181}{3, 1}{2, 5}{2, 2}{\dfrac{12 \sqrt{2}}{\Gamma (\frac{1}{2})}\dfrac{2^{3 n/2}}{n^{3/2}}}%
 \item{3}{A151255}{\stepset10001100}{1, 1, 2, 3, 8, 15, 39, 77}{6, 8}{4, 16}{--}{\dfrac{24 \sqrt{2}}{\pi }\dfrac{2^{3 n/2}}{n^2}}%
 \item{3}{A151265}{\stepset01010001}{1, 1, 3, 7, 17, 47, 125, 333}{6, 4}{4, 9}{6, 8}{\dfrac{2 \sqrt{2}}{\Gamma (\frac{1}{4})}\dfrac{3^n}{n^{3/4}}}%
 \item{3}{A151266}{\stepset00110001}{1, 1, 3, 7, 19, 49, 139, 379}{7, 10}{5, 16}{--}{\dfrac{\sqrt{3}}{2\Gamma (\frac{1}{2})}\dfrac{3^n}{\sqrt{n}}}%
 \item{3}{A151278}{\stepset10001010}{1, 2, 4, 10, 26, 66, 178, 488}{7, 4}{4, 12}{6, 8}{\dfrac{3 \sqrt{3}}{\sqrt{2}\Gamma (\frac{1}{4})}\dfrac{3^n}{n^{3/4}}}%
 \item{3}{A151281}{\stepset00001101}{1, 2, 6, 16, 48, 136, 408, 1184}{3, 1}{2, 5}{2, 2}{\dfrac{1}{2}3^n}%
 \item{4}{A005558}{\stepset00111100}{1, 1, 3, 6, 20, 50, 175, 490}{2, 3}{3, 5}{--}{\dfrac{8}{\pi }\dfrac{4^n}{n^2}}%
 \item{4}{A005566}{\stepset01011010}{1, 2, 6, 18, 60, 200, 700, 2450}{2, 2}{3, 4}{--}{\dfrac{4}{\pi }\dfrac{4^n}{n}}%
 \item{4}{A018224}{\stepset10100101}{1, 1, 4, 9, 36, 100, 400, 1225}{2, 3}{3, 5}{--}{\dfrac{2}{\pi }\dfrac{4^n}{n}}%
 \item{4}{A060899}{\stepset00011101}{1, 2, 8, 24, 96, 320, 1280, 4480}{2, 1}{2, 3}{2, 2}{\dfrac{\sqrt{2}}{\Gamma (\frac{1}{2})}\dfrac{4^n}{\sqrt{n}}}%
 \item{4}{A060900}{\stepset10011001}{1, 2, 7, 21, 78, 260, 988, 3458}{2, 3}{3, 5}{8, 9}{\dfrac{4 \sqrt{3}}{3\Gamma (\frac{1}{3})}\dfrac{4^n}{n^{2/3}}}%
 \item{4}{A128386}{\stepset10010101}{1, 1, 4, 7, 28, 58, 232, 523}{3, 1}{2, 5}{2, 2}{\dfrac{6 \sqrt{2}}{\Gamma (\frac{1}{2})}\dfrac{2^n 3^{n/2}}{n^{3/2}}}%
 \item{4}{A129637}{\stepset00001111}{1, 3, 11, 41, 157, 607, 2367, 9277}{3, 1}{2, 5}{2, 2}{\dfrac{1}{2}4^n}%
 \item{4}{A151261}{\stepset10011100}{1, 1, 3, 5, 17, 34, 121, 265}{5, 8}{4, 15}{--}{\dfrac{12 \sqrt{3}}{\pi }\dfrac{2^n 3^{n/2}}{n^2}}%
 \item{4}{A151282}{\stepset00010111}{1, 2, 6, 18, 58, 190, 638, 2170}{3, 1}{2, 5}{2, 2}{\dfrac{A^2 B^{3/2}}{2^{3/4}\Gamma (\frac{1}{2})}\dfrac{B^n}{n^{3/2}}}%
 \item{4}{A151291}{\stepset00111001}{1, 2, 7, 23, 84, 301, 1127, 4186}{6, 10}{5, 15}{--}{\dfrac{4}{3\Gamma (\frac{1}{2})}\dfrac{4^n}{\sqrt{n}}}%
 \item{5}{A151275}{\stepset10110101}{1, 1, 5, 13, 61, 199, 939, 3389}{9, 18}{5, 24}{--}{\dfrac{12 \sqrt{30}}{\pi }\dfrac{(\sqrt{24})^n}{n^2}}%
 \item{5}{A151287}{\stepset10111010}{1, 2, 6, 21, 76, 290, 1148, 4627}{7, 11}{5, 19}{--}{\dfrac{\sqrt{8} A^{7/2}}{\pi }\dfrac{(2A)^n}{n^2}}%
 \item{5}{A151292}{\stepset10010111}{1, 2, 7, 23, 85, 314, 1207, 4682}{3, 1}{2, 5}{2, 2}{\dfrac{\sqrt[4]{3} C^2 D^{3/2}}{8\Gamma (\frac{1}{2})}\dfrac{D^n}{n^{3/2}}}%
 \item{5}{A151302}{\stepset10100111}{1, 2, 8, 29, 129, 535, 2467, 10844}{9, 18}{5, 24}{--}{\dfrac{\sqrt{5}}{3 \sqrt{2}\Gamma (\frac{1}{2})}\dfrac{5^n}{\sqrt{n}}}%
 \item{5}{A151307}{\stepset01011101}{1, 2, 9, 34, 151, 659, 2999, 13714}{8, 15}{5, 20}{--}{\dfrac{\sqrt{5}}{2 \sqrt{2}\Gamma (\frac{1}{2})}\dfrac{5^n}{\sqrt{n}}}%
 \item{5}{A151318}{\stepset00011111}{1, 3, 13, 55, 249, 1131, 5253, 24543}{2, 1}{2, 3}{2, 2}{\dfrac{\sqrt{5/2}}{\Gamma (\frac{1}{2})}\dfrac{5^n}{\sqrt{n}}}%
 \item{6}{A129400}{\stepset01111110}{1, 2, 8, 32, 144, 672, 3264, 16256}{2, 1}{2, 3}{2, 2}{\dfrac{3 \sqrt{3}}{2\Gamma (\frac{1}{2})}\dfrac{6^n}{n^{3/2}}}%
 \item{6}{A151297}{\stepset11011110}{1, 2, 7, 26, 105, 444, 1944, 8728}{7, 11}{5, 18}{--}{\dfrac{\sqrt{3} C^{7/2}}{2\pi } \dfrac{(2 C)^n}{n^2}}%
 \item{6}{A151312}{\stepset10111101}{1, 2, 10, 39, 210, 960, 5340, 26250}{4, 5}{3, 8}{--}{\dfrac{\sqrt{6}}{\pi }\dfrac{6^n}{n}}%
 \item{6}{A151323}{\stepset11011011}{1, 3, 14, 67, 342, 1790, 9580, 52035}{2, 1}{2, 3}{4, 4}{\dfrac{\sqrt{2}\, 3^{3/4}}{\Gamma (\frac{1}{4})}\dfrac{6^n}{n^{3/4}}}%
 \item{6}{A151326}{\stepset01011111}{1, 3, 15, 74, 392, 2116, 11652, 64967}{7, 14}{5, 18}{--}{\dfrac{2 \sqrt{3}}{3\Gamma (\frac{1}{2})}\dfrac{6^n}{\sqrt{n}}}%
 \item{7}{A151314}{\stepset11110111}{1, 2, 11, 49, 277, 1479, 8679, 49974}{9, 18}{5, 24}{--}{\dfrac{EF^{7/2}}{5\sqrt{95}\pi } \dfrac{(2F)^n}{n^2}}%
 \item{7}{A151329}{\stepset10111111}{1, 3, 16, 86, 509, 3065, 19088, 120401}{9, 18}{5, 24}{--}{\dfrac{\sqrt{7/3}}{3\Gamma (\frac{1}{2})}\dfrac{7^n}{\sqrt{n}}}%
\global\let\\\empty
 \item{8}{A151331}{\stepset11111111}{1, 3, 18, 105, 684, 4550, 31340, 219555}{3, 4}{3, 6}{--}{\dfrac{8}{3\pi }\dfrac{8^n}{n}}%
\end{longtable}%

\bigskip

\normalsize
\mytable\label{tab:3dalg} Conjecturally algebraic series and their step sets in~3D.
Step set figures are as in Section~\ref{sec:4}. Equation sizes are as in Table~\ref{tab:2d}.

\medskip

\let\stepset\Stepset

\footnotesize
\def\item#1#2#3#4#5#6#7#8{#3, \dots\quad(#2)&#4&#5&#6&#7&#8\\}%
\begin{longtable}{l|c|c|c|c|c}%
 First terms\quad(OEIS Tag) & \multicolumn{2}{|c}{Step sets} & \multicolumn{3}{|c}{Equation sizes} \\\hline \endhead
 \item{7}{A025237}{1, 1, 4, 10, 37, 121, 451, 1639}{\stepset00001000000000001100100100}{\stepset00000001000000001100100100}{2, 1}{2, 3}{2, 2}%
 \item{7}{A149576}{1, 1, 5, 15, 51, 199, 755, 2789}{\stepset00001000010010100000000001}{\stepset00000001010010100000000001}{11, 22}{7, 31}{12, 17}%
 \item{7}{A149847}{1, 2, 4, 14, 46, 134, 502, 1820}{\stepset10010010000000001000010000}{\stepset10010010000000001000000010}{8, 6}{4, 16}{6, 9}%
\end{longtable}%

\bigskip

\normalsize
\mytable\label{tab:3dtrans} Conjecturally transcendental D-finite generating series and their step sets in~3D. 
The equation sizes columns refer to (minimal) recurrence equations, and differential equations, respectively.
\medskip

\def\itemA#1#2#3#4#5#6#7#8{#2&#4&#6&#7&\global\let\item\itemB}%
\def\itemB#1#2#3#4#5#6#7#8{#2&#4&#6&#7\\\global\let\item\itemA}%
\let\item\itemA
\footnotesize
\begin{longtable}{c|c|c|c||c|c|c|c}%
 OEIS Tag & Step sets & \multicolumn{2}{|c||}{Equation sizes} & OEIS Tag & Step sets & \multicolumn{2}{|c}{Equation sizes} \\\hline 
 \item{4}{A148060}{1, 1, 2, 3, 12, 25, 77, 161}{\stepset10010010000000001100000000}{}{9, 17}{5, 28}{--}   
   \item{3}{A148438}{1, 1, 2, 6, 15, 43, 143, 437}{\stepset00000000110010100000010000}{}{7, 10}{5, 17}{--}%
 \item{}{}{}{\stepset10010010000000001000100000}{}{}{}{}%
   \item{}{}{}{\stepset00000100010010100000000010}{}{}{}{}%
 \item{}{}{}{\stepset10010010100000000000010000}{}{}{}{}%
   \item{}{}{}{\stepset00000000110010100000000010}{}{}{}{}%
 \item{}{}{}{\stepset10010010100000000000000010}{}{}{}{}%
   \item{3}{A149090}{1, 1, 4, 7, 34, 73, 349, 817}{\stepset10000000000000001100100100}{}{9, 17}{5, 28}{--}%
 \item{2}{A149589}{1, 1, 5, 15, 57, 205, 809, 3119}{\stepset00001000000000000100100101}{}{10, 21}{6, 29}{--}%
   \item{}{}{}{\stepset00110100100000000000000010}{}{}{}{}%
 \item{}{}{}{\stepset00000001000000000100100101}{}{}{}{}%
   \item{}{}{}{\stepset00000010000000001100100100}{}{}{}{}%
 \item{1}{A005817}{1, 1, 2, 4, 10, 25, 70, 196}{\stepset00000100000010010010000000}{}{2, 2}{3, 4}{--}%
 \item{1}{A148005}{1, 1, 2, 3, 8, 15, 44, 91}{\stepset00000010010001000000100000}{}{5, 8}{4, 15}{--}%
 \item{1}{A148052}{1, 1, 2, 3, 10, 20, 63, 133}{\stepset00000010010001100100000000}{}{7, 18}{6, 27}{--}%
 \item{1}{A148068}{1, 1, 2, 3, 12, 25, 87, 189}{\stepset10000010000001000100000100}{}{7, 17}{6, 25}{--}%
 \item{1}{A148072}{1, 1, 2, 4, 9, 21, 56, 148}{\stepset00000100010000100000010000}{}{12, 57}{10, 69}{--}%
 \item{1}{A148162}{1, 1, 2, 4, 11, 31, 91, 267}{\stepset00001000000001000100000100}{}{4, 3}{3, 6}{--}%
 \item{1}{A148284}{1, 1, 2, 5, 12, 32, 97, 282}{\stepset00000100010010100000010000}{}{14, 57}{10, 71}{--}%
 \item{1}{A148331}{1, 1, 2, 5, 14, 42, 137, 464}{\stepset00000010110100000000010000}{}{11, 43}{9, 53}{--}%
 \item{1}{A148507}{1, 1, 3, 5, 17, 34, 126, 279}{\stepset00000010010011000000100000}{}{4, 6}{4, 11}{--}%
 \item{1}{A148525}{1, 1, 3, 5, 19, 39, 155, 349}{\stepset00010000010001100000100000}{}{7, 16}{6, 25}{--}%
 \item{1}{A148548}{1, 1, 3, 5, 21, 44, 179, 405}{\stepset10000000010001100000000100}{}{7, 19}{6, 28}{--}%
 \item{1}{A148689}{1, 1, 3, 7, 23, 64, 223, 687}{\stepset00000100010000100000001000}{}{8, 25}{8, 31}{--}%
 \item{1}{A148703}{1, 1, 3, 7, 23, 71, 246, 848}{\stepset00001000000001000100100100}{}{4, 3}{3, 6}{--}%
 \item{1}{A148790}{1, 1, 3, 8, 25, 77, 257, 853}{\stepset00000000100110000000001000}{}{6, 12}{5, 18}{--}%
 \item{1}{A148934}{1, 1, 3, 9, 28, 100, 365, 1365}{\stepset00000001010100000000101000}{}{5, 5}{4, 11}{--}%
 \item{1}{A149279}{1, 1, 4, 11, 44, 133, 585, 2067}{\stepset00000100010010100000001000}{}{14, 62}{10, 75}{--}%
 \item{1}{A149290}{1, 1, 4, 11, 45, 166, 690, 2855}{\stepset00001000000000101101000000}{}{11, 53}{9, 61}{--}%
 \item{1}{A149363}{1, 1, 4, 12, 44, 160, 635, 2520}{\stepset00000000100110001001000000}{}{7, 16}{6, 24}{--}%
 \item{1}{A149632}{1, 1, 5, 15, 69, 217, 1061, 3923}{\stepset00000000110010100000000001}{}{7, 11}{5, 16}{--}%
 \item{1}{A149713}{1, 1, 5, 17, 71, 289, 1269, 5529}{\stepset00001000000000000101000101}{}{8, 22}{7, 29}{--}%
 \item{1}{A150054}{1, 2, 6, 18, 62, 215, 809, 3045}{\stepset00001000010001100000010000}{}{12, 39}{9, 52}{--}%
 \item{1}{A150370}{1, 2, 7, 23, 94, 366, 1572, 6510}{\stepset00000100010001100000001000}{}{14, 62}{10, 75}{--}%
 \item{1}{A150410}{1, 2, 7, 24, 94, 370, 1537, 6440}{\stepset00000000100111000000001000}{}{4, 6}{4, 11}{--}%
 \item{1}{A150471}{1, 2, 7, 25, 99, 402, 1687, 7242}{\stepset00001000001010001000010000}{}{12, 33}{8, 42}{--}%
 \item{1}{A150499}{1, 2, 7, 25, 101, 414, 1773, 7680}{\stepset00000001001011000000000010}{}{14, 48}{9, 61}{--}%
 \item{1}{A150764}{1, 2, 8, 30, 126, 530, 2330, 10290}{\stepset00000100000110001000001000}{}{7, 13}{6, 19}{--}%
 \item{1}{A150950}{1, 2, 9, 35, 155, 677, 3095, 14118}{\stepset00001000000000000100101001}{}{8, 23}{7, 29}{--}%
 \global\let\\\empty
 \item{1}{A151053}{1, 3, 10, 37, 144, 586, 2454, 10491}{\stepset00001000010001010000010000}{}{14, 38}{9, 48}{--}%
\end{longtable}%

\end{document}